\documentclass[a4paper,11pt]{article}
\usepackage{amsfonts,amssymb,amsmath, dsfont,relsize,accents}
\usepackage{bm}
\usepackage{mathrsfs}  % curly fonts
\usepackage{color}
\usepackage{upgreek}
\usepackage[english]{babel}
\usepackage{graphicx}
\usepackage{microtype}
\usepackage[colorlinks=true, pdfstartview=FitV, linkcolor=blue, citecolor=blue, urlcolor=blue,pagebackref=false]{hyperref}
\usepackage{graphicx}
\usepackage{epstopdf}
\definecolor{labelkey}{gray}{.8}
\definecolor{refkey}{gray}{.8}

\definecolor{darkred}{rgb}{0.9,0.1,0.1}

 \makeatletter
 \@addtoreset{equation}{section}
 \makeatother

 \makeatletter
 \@addtoreset{enunciato}{section}
 \makeatother

 \newcounter{enunciato}[section]

 \newtheorem{ittheorem}{Theorem}
 \newtheorem{itlemma}{Lemma}
 \newtheorem{itproposition}{Proposition}
 \newtheorem{itcorollary}{Corollary}
 \newtheorem{itdefinition}{Definition}
 \newtheorem{itremark}{Remark}
 \newtheorem{itclaim}{Claim}
 \newtheorem{itfact}{Fact}
 \newtheorem{itconjecture}{Conjecture}

 \newenvironment{theorem}{\addtocounter{enunciato}{1}
 \begin{ittheorem}}{\end{ittheorem}}

 \newenvironment{lemma}{\addtocounter{enunciato}{1}
 \begin{itlemma}}{\end{itlemma}}

 \newenvironment{proposition}{\addtocounter{enunciato}{1}
 \begin{itproposition}}{\end{itproposition}}

 \newenvironment{corollary}{\addtocounter{enunciato}{1}
 \begin{itcorollary}}{\end{itcorollary}}

 \newenvironment{definition}{\addtocounter{enunciato}{1}
 \begin{itdefinition}}{\end{itdefinition}}

 \newenvironment{remark}{\addtocounter{enunciato}{1}
 \begin{itremark}}{\end{itremark}}

 \newenvironment{claim}{\addtocounter{enunciato}{1}
 \begin{itclaim}}{\end{itclaim}}

 \newenvironment{fact}{\addtocounter{enunciato}{1}
 \begin{itfact}}{\end{itfact}}

 \newenvironment{conjecture}{\addtocounter{enunciato}{1}
 \begin{itconjecture}}{\end{itconjecture}}

 \newcommand{\be}[1]{\begin{equation}\label{#1}}
 \newcommand{\ee}{\end{equation}}

 \newcommand{\bl}[1]{\begin{lemma}\label{#1}}
 \newcommand{\el}{\end{lemma}}

 \newcommand{\br}[1]{\begin{remark}\label{#1}}
 \newcommand{\er}{\end{remark}}

 \newcommand{\bt}[1]{\begin{theorem}\label{#1}}
 \newcommand{\et}{\end{theorem}}

 \newcommand{\bd}[1]{\begin{definition}\label{#1}}
 \newcommand{\ed}{\end{definition}}

 \newcommand{\bcl}[1]{\begin{claim}\label{#1}}
 \newcommand{\ecl}{\end{claim}}

 \newcommand{\bfact}[1]{\begin{fact}\label{#1}}
 \newcommand{\efact}{\end{fact}}

 \newcommand{\bp}[1]{\begin{proposition}\label{#1}}
 \newcommand{\ep}{\end{proposition}}

 \newcommand{\bc}[1]{\begin{corollary}\label{#1}}
 \newcommand{\ec}{\end{corollary}}

 \newcommand{\bcj}[1]{\begin{conjecture}\label{#1}}
 \newcommand{\ecj}{\end{conjecture}}

 \newcommand{\bpr}{\begin{proof}}
 \newcommand{\epr}{\end{proof}}

 \newcommand{\bprlem}[1]{\begin{proofof}{\it Lemma \ref{#1}}.\,\,}
 \newcommand{\eprlem}{\end{proofof}}

 \newcommand{\bprthm}[1]{\begin{proofof}{\it Theorem \ref{#1}}.\,\,}
 \newcommand{\eprthm}{\end{proofof}}

 \newcommand{\bprprop}[1]{\begin{proofof}{\it Proposition \ref{#1}}.\,\,}
 \newcommand{\eprprop}{\end{proofof}}

 \newcommand{\bi}{\begin{itemize}}
 \newcommand{\ei}{\end{itemize}}

 \newcommand{\ben}{\begin{enumerate}}
 \newcommand{\een}{\end{enumerate}}

 \newenvironment{proof}{\noindent {\em Proof}.\,\,}{\hspace*{\fill}$\halmos$\medskip}
 \newenvironment{proofof}{\noindent {\em Proof of\,\,}}{\hspace*{\fill}$\halmos$\medskip}
 \newcommand{\halmos}{\rule{1ex}{1.4ex}}

 \newcommand{\one}{{\mathchoice {1\mskip-4mu\mathrm l}
         {1\mskip-4mu\mathrm l}
         {1\mskip-4.5mu\mathrm l}
         {1\mskip-5mu\mathrm l}}}

\def \E {{\mathbb E}}

\def \N {{\mathbb N}}
\def \P {{\mathbb P}}

\def \R {{\mathbb R}}
\def \Z {{\mathbb Z}}

\def \lra \leftrightarrow

\def \ra {\rightarrow}
\def \ba {\begin{array}}
\def \ea {\end{array}}

\def \lra {\longrightarrow}

\def \cI {{\mathcal I}}
\def \cJ {{\mathcal J}}

\def \lra {{\leftrightarrow}}

\def \Ll {\left}
\def \Rr {\right}
\def \subset {\subseteq}

\def \emptyset {\varnothing}

\def\one{\rlap{\mbox{\small\rm 1}}\kern.15em 1}

\newlength{\dhatheight}

\begin{document}
\title{Extinction time for the contact process on general graphs}

\author{Bruno Schapira\textsuperscript{1}, Daniel Valesin\textsuperscript{2}}
\footnotetext[1]{Aix-Marseille Universit\'e, CNRS, Centrale Marseille, I2M, UMR 7373, 13453 Marseille, France.\\ \url{bruno.schapira@univ-amu.fr}}
\footnotetext[2]{\noindent University of Groningen, Nijenborgh 9, 9747 AG Groningen, The Netherlands.\\ \url{d.rodrigues.valesin@rug.nl}}
\date{September 10, 2015}
\maketitle

\begin{abstract}
We consider the contact process on finite and connected graphs and study the behavior of the extinction time, that is, the amount of time that it takes for the infection to disappear in the process started from full occupancy. We prove, without any restriction on the graph $G$, that if the infection rate $\lambda$ is larger than the critical rate of the one-dimensional process, then the extinction time grows faster than $\exp\{|G|/(\log|G|)^\kappa\}$ for any constant $\kappa > 1$, where $|G|$ denotes the number of vertices of $G$. Also for general graphs, we show that the extinction time divided by its expectation converges in distribution, as the number of vertices tends to infinity, to the exponential distribution with parameter 1. These results complement earlier work of Mountford, Mourrat, Valesin and Yao, in which only graphs of bounded degrees were considered, and the extinction time was shown to grow exponentially in $n$; here we also provide a simpler proof of this fact.\end{abstract}

\section{Introduction}
The \textit{contact process} $(\xi_t)_{t \geq 0}$ with \textit{infection rate} $\lambda$ on a graph $G = (V,E)$ is the Markov process on the space $\{0,1\}^V$ and generator given, for any cylindrical function $f$, by
\begin{equation}\label{eq:generator}\mathcal{L}f(\xi) = \sum_{x \in V:\xi(x) =1} \left( (f(\xi^{x\leftarrow 0}) - f(\xi)) + \lambda\sum_{y \in V: y \sim x}(f(\xi^{y \leftarrow 1}) - f(\xi))\right),\end{equation}
where $y \sim x$ means that $x$ and $y$ are neighbors and $\xi^{z \leftarrow i}$, for $z \in V$ and $i \in \{0,1\}$, is the configuration defined by $\xi^{z \leftarrow i}(z) = i$ and $\xi^{z \leftarrow i}(x) = \xi(x)$ for any $x \neq z$. Vertices of the graph are interpreted as individuals in a population; each individual can be healthy (state 0) or infected (state 1). The above generator prescribes that infected individuals become healthy with rate 1 and transmit the infection to each neighbor with rate $\lambda$.

We denote by $\underline{0}$ and $\underline{1}$ the elements of $\{0,1\}^V$ that are identically equal to 0 and 1, respectively. Inspecting the above generator shows that $\underline{0}$ is an absorbing state for the dynamics. Let $x \in V$ and $A \subset V$; we denote by $(\xi^x_t)$, $(\xi^A_t)$ and $(\xi^{\underline{1}}_t)$ the process started from $\mathds{1}_{\{x\}}$, $\mathds{1}_A$ and $\underline{1}$, respectively ($\mathds{1}$ is the indicator function). We also denote by $\P_\lambda$ a probability measure under which the contact process with rate $\lambda$ is defined on the graph $G$ (which will be clear from the context, as will the initial configuration of the process); later we will fix $\lambda$ and omit it from the notation as well. We denote by $\E_\lambda$, or sometimes simply $\E$, the associated expectation.

In \cite{lig2}, the reader can find a thorough introduction to the contact process. For the sake of the remainder of this introduction, let us say a few words about its \textit{phase transition}, starting with the case $G = \Z^d$, the $d$-dimensional integer lattice. Define the following \textit{survival events}:
\begin{equation*} \label{eq:events}S_\text{global} := \{\xi^0_t \neq \underline{0} \text{ for all }t\}\supseteq  \{\text{for all } t_0 \text{ there exists } t_1 > t_0: \xi^0_{t_1}(0) = 1\} =:S_{\text{local}}.\end{equation*}
Then, there exists $\lambda_c = \lambda_c(\Z^d) > 0$ so that: if $\lambda \leq \lambda_c$, then $\P_\lambda[S_\text{global}] = 0$ and if $\lambda > \lambda_c$, then $\P_\lambda\Ll[S_\text{global}\Rr] > 0$ and $\P_\lambda\Ll[S_\text{local}\mid S_\text{global}\Rr] = 1$. Now take $G = \mathbb{T}^d$, the infinite regular tree of degree $d\geq 3$, fix a root vertex and denote it by $0$, and take the same survival events as defined above. Then, there exist $\lambda_c^{(1)}= \lambda_c^{(1)}(\mathbb{T}^d)$, $\lambda_c^{(2)} = \lambda_c^{(2)}(\mathbb{T}^d)$ so that $0 <\lambda_c^{(1)}< \lambda_c^{(2)} < \infty$ and: if $\lambda \leq \lambda_c^{(1)}$, then $\P_\lambda\Ll[S_\text{global}\Rr] = 0$; if $\lambda_c^{(1)} < \lambda \leq\lambda_c^{(2)}$, then $\P_\lambda\Ll[S_\text{global}\Rr] >0$ and  $\P_\lambda\Ll[S_\text{local}\Rr] = 0$; if $\lambda > \lambda_c^{(2)}$, then $\P_\lambda\Ll[S_\text{global}\Rr] >0$ and  $\P_\lambda\Ll[S_\text{local}\mid S_\text{global}\Rr] = 1$.

In case $G$ is a finite graph, we have $\P_\lambda[S_\text{global}] = \P_\lambda[S_\text{local}] = 0$, since the process is then a continuous-time Markov chain with a finite state space and the trap $\underline{0}$ can be reached from any other configuration; in particular the \textit{extinction time}
$$\uptau_G = \inf\{t: \xi^{\underline{1}}_t = \underline{0}\}$$
is necessarily finite. Hence, on finite graphs there can be no phase transition in the sense presented in the previous paragraph. Still, one can study the dependence of the process on the value of $\lambda$, and in some cases make sense of a finite-volume phase transition. This project typically goes as follows: one fixes $\lambda > 0$ and some sequence of graphs $(G_n)_{n \geq 1}$ (usually converging or increasing, in some sense, to an infinite graph, or belonging to some class of random graphs), and then studies the asymptotic behavior of the random variables $\uptau_{G_n}$, including their dependence on $\lambda$. This has been carried out in the case of boxes of $\Z^d$ (\cite{CGOV84}, \cite{sc85}, \cite{DL88}, \cite{Ch94}, \cite{DS88}, \cite{mo93}, \cite{mo99}), finite homogeneous trees (\cite{St01}, \cite{CMMV13}), the configuration model (\cite{CD09}, \cite{mvy13}, \cite{haobruno}, \cite{mv14}, \cite{LS15}) and the preferential attachment graph (\cite{BBCS05}, \cite{hao}).

These are successful case studies, but they of course depend on exploring the structure of the graphs under consideration and sometimes their relation to some infinite (possibly random) graph. In contrast, one may  wonder if there are results that are context-free, that is, that hold for \textit{arbitrary} sequences of graphs. Indeed, the following facts have been established. Given a graph $G$, let $|G|$ denote its number of vertices.
\begin{theorem}\label{thm:ext_bound}
\begin{itemize}
\item[(i)] \cite{mv14} For any $d \in \N$ and $\lambda < \lambda_c^{(1)}(\mathbb{T}^d)$ there exists $C > 0$ such that, for any graph $G$ with degree bounded by $d$ and at least two vertices,
$$\E_\lambda[\uptau_G] \leq C\log(|G|).$$
\item[(ii)]\cite{mmvy13} For any $d \in \N$ and $\lambda > \lambda_c(\Z)$ there exists $c > 0$ such that, for any connected graph $G$ with degree bounded by $d$ and at least two vertices,
$$\E_\lambda[\uptau_G] \geq \exp\{c|G|\}.$$
\end{itemize}
\end{theorem}
Our motivation in this paper is to improve the second part of Theorem \ref{thm:ext_bound}. With the generality that the result is stated, the restriction that $\lambda > \lambda_c(\Z)$ cannot be relaxed:  the class of graphs under consideration includes line segments of $\Z$ and for those, the extinction time grows logarithmically with the number of vertices when $\lambda < \lambda_c(\Z)$.  In contrast, the requirement that the degree be bounded seems unnecessary: if vertices of larger and larger degree are present, this should only contribute to the extinction time being larger. The reason this requirement was present in \cite{mmvy13} was a technical convenience: it allowed for the application of a certain lemma concerning the splitting of trees into large subtrees (this lemma is reproduced here: see Lemma \ref{lem:split} below). Our main result is:
\begin{theorem}\label{thm:ext}
For any $\lambda > \lambda_c(\Z)$ and any $\varepsilon>0$, there exists a constant $c_\varepsilon$ such that for any connected graph $G$ with at least two vertices,
\begin{equation}\label{eq:main_thm}\mathbb{E}_\lambda[\uptau_G] \ge  \exp\left\{c_\varepsilon\,  \frac {|G|}{(\log |G|)^{1+\varepsilon}}\right\}\end{equation}
and, for any non-empty $A \subset G$,
\begin{equation}
\P\left[\xi^A_{\exp\left\{c_\varepsilon |G|/(\log |G|)^{1+\varepsilon}\right\}} \neq \underline{0}\right] > c_\varepsilon.\label{eq:main_thm2}
\end{equation}
\end{theorem}

This theorem, as well as the second part of Theorem \ref{thm:ext_bound}, imply that any sequence of graphs has a ``supercritical phase'', which contains the parameter values $\lambda\in (\lambda_c(\mathbb{Z}),\infty)$. This is certainly informative, but in many specific cases $\lambda_c(\Z)$ is not the optimal threshold; for example, if $G_n$ is given by increasing boxes of $\Z^d$ with $d$ large, then the extinction time grows exponentially if $\lambda > \lambda_c(\Z^d)$, which is smaller than $\lambda_c(\Z)$. More drastically, in some graphs with unbounded degree, such as the configuration model with power law degree distribution or the preferential attachment graph, the extinction time grows as an exponential (or at least  stretched exponential) function of $|G_n|$ for \textit{any} positive $\lambda$.

In spite of not directly giving the optimal rate in specific cases, Theorem \ref{thm:ext_bound} (ii) and Theorem \ref{thm:ext} can be useful in the process of obtaining the optimal rate. For one thing, our proof of Theorem \ref{thm:ext} is versatile in that it relies on quite useful inequalities and simple methods and could easily be adapted to other contexts (see below for a discussion of our strategy of proof). In addition, lower bounds on the extinction time often follow from some type of coarse graining or renormalization procedure in which, by partitioning space and time into large units, one obtains a new version of the process, in which a notion of infection rate can also be made precise and can often be made as large as desired. An instance of this is found in \cite{mmvy13}, where Theorem \ref{thm:ext_bound} is used in the treatment of the contact process on a graph given by the configuration model with a power law degree distribution.

We also prove:
\begin{theorem}\label{thm:couple}
For any $\lambda > \lambda_c(\Z)$ and any sequence of graphs $(G_n)_{n \in \N}$ with $|G_n| \to \infty$ as $n \to \infty$,
$$\frac{\uptau_{G_n}}{\mathbb{E}_\lambda[\uptau_{G_n}]} \xrightarrow[\text{(d.)}]{n \to \infty} \mathrm{Exp}(1).$$
\end{theorem}
This is a generalization of Theorem 1.2 of \cite{mmvy13}, which is the same statement with a bounded degree assumption.

Let us make some comments on the proofs of these results now. 
Our main tool is a completely general coupling result, Proposition \ref{prop:prelim_coup}, which shows that on any graph, 
if the process starting from a single vertex survives for a time comparable to the size of the graph, then with high probability 
it couples with (meaning that it is equal to) the process starting from full occupancy. It is well-known that this, together with a mild
lower bound on the extinction time, already implies Theorem \ref{thm:couple}. 
Another important consequence is Proposition \ref{lem:induction} which asserts 
that for any decomposition of a graph into disjoint components (or subgraphs), the mean extinction time on the original graph is larger than 
the product of the mean extinction times on these subgraphs, up to some correction term. This term remains negligible as long as the number of components in the decomposition is not too large. 
Such a result is of course particularly well suited for proofs 
going by induction on the size of the graph, specially for proving exponential (or almost exponential) lower bounds, in virtue of the formula $\exp(x+y)=\exp(x)\exp(y)$. 

With Proposition \ref{lem:induction} at hand, we prove Theorem \ref{thm:ext} and also give a new proof of Theorem \ref{thm:ext_bound} (ii), simpler than the one in \cite{mmvy13}. Since in Theorem \ref{thm:ext_bound} (ii) it is assumed that the degrees are bounded, one only needs to split the graph in a bounded number of pieces, independently of the size of the graph, so that the correction term in Proposition \ref{lem:induction} causes no problem, and we get a true exponential lower bound (a similar proof was used in \cite{CMMV13} in the setting of finite regular trees). 
However, for general graphs, the number of pieces required in the decomposition might be very large, making 
the correction term explode, and this explains why we have the logarithmic term in Theorem \ref{thm:ext}.

Now the paper is organized as follows. Section 2 contains all the material preparing to the proofs of the main results. 
In particular in Subsection \ref{sec.not} we recall some standard definitions and fix some notation. 
In subsection \ref{sec.basictool}, we give some basic tools, among which some preliminary estimates for the contact process on a line segment and a star graph. 
In Subsection \ref{sec.maintool} we state and prove the main tools discussed above, namely the coupling result, Proposition \ref{prop:prelim_coup}, and Proposition \ref{lem:induction}. Then Section 3 contains the actual proofs of the main results. It is organized as follows. 
We first give in Subsection \ref{sec.lev1} a mild polynomial lower bound. As we already mentioned, together with the coupling lemma, this implies Theorem \ref{thm:couple}: we explain this in slightly more details in Subsection \ref{sec.couple}.   In Subsection \ref{sec.lev2} we prove a stretched exponential lower bound, which is a necessary 
intermediate step toward the proof of Theorem \ref{thm:ext}. 
In Subsection \ref{sec.extbound} we explain how one can also deduce Theorem \ref{thm:ext_bound} (ii),  
by using induction on the size of the graph.
Finally the full proof of Theorem \ref{thm:ext} is given in Subsections \ref{sec.lev3} and \ref{sec:lev4} where we put all pieces together.

\section{Preliminary results and tools}
\label{s:prelim}

\subsection{Notation and definitions} \label{sec.not}
A \textit{graph} will be understood as a set $V$ of vertices and a set $E \subset \{\{x,y\} \subset V: x \neq y\}$ of edges. Thus, for convenience we will not explicitly treat graphs with loops (edges that start and end at the same vertex) and parallel edges between vertices, though one can define the contact process on those graphs as well and our results could then be readily adapted. The graphs we consider will always be connected. We denote by $|G|$ the number of vertices of $G$; for a set $A$, we denote by $|A|$ the number of elements of $A$. We will often abuse notation and identify a graph with its set of vertices; for example, we may write $\{0,1\}^G$ in place of $\{0,1\}^V$.

\begin{remark}
For several of our results, it is sufficient to give a proof for trees only. For example, if Theorem \ref{thm:ext} is proved for trees and we then consider a general graph $G$, we can apply the result to an arbitrary spanning tree $T$ of $G$ and observe that the contact process on $T$ is dominated (in the natural stochastic order of configurations) by the contact process on $G$, hence the extinction time of the latter is larger. We will not repeat this sufficiency in every situation in which it applies.
\end{remark}

From now on, we fix a value $\lambda > \lambda_c(\Z)$ and will omit it from the notation. In particular, many of the constants we define below may depend on $\lambda$. In order to fix notation, we will quickly go over the very well-known \textit{graphical construction} of the contact process. Fixing $G = (V,E)$, we take a family of independent Poisson point processes on $[0,\infty)$,
$$(D^x)_{x \in V} \text{ each with rate 1,}\qquad (D^{(x,y)})_{x,y \in V: \{x,y\} \in E} \text{ each with rate } \lambda.$$
Such a family is called a \textit{Harris system}. We view each of these processes as a random discrete subset of $[0,\infty)$. Arrivals of the processes $(D^x)$ are called \textit{recovery marks}, and arrivals of the processes $(D^{(x,y)})$ are called \textit{transmissions}. Given $x,y\in V$ and $0 \leq s < t$, an \textit{infection path} from $(x,s)$ to $(y,t)$ is a function $\gamma:[s,t] \to V$ such that
$$\gamma(s) = x,\quad \gamma(t) = y,\quad s \notin D^{\gamma(s)} \text{ for all } s \quad\text{ and } s \in D^{(\gamma(s-), \gamma(s))} \text{ whenever } \gamma(s-) \neq \gamma(s). $$
If such a path exists, we say $(x,s)$ and $(y,t)$ are connected by an infection path, and write $(x,s) \leftrightarrow (y,t)$. We convention to put $(x,s) \leftrightarrow (x,s)$. For $A \subset V$ and $I \subset [0,t]$, we write $A \times I \leftrightarrow (y,t)$ if $(x,s) \leftrightarrow (y,t)$ for some $x \in A$ and $s \in I$; similarly we write $(x,s) \leftrightarrow B \times J$ and $A \times I \leftrightarrow B \times J$. Finally, given $C \subset V$, we write $(x,s) \stackrel{C}{\leftrightarrow} (y,t)$ if there exists an infection path from $(x,s)$ to $(y,t)$ that is entirely contained in $C$. (Similarly, we write $A \times I \stackrel{C}{\leftrightarrow} (y,t)$, $(x,s) \stackrel{C}{\leftrightarrow} B \times J$ and $A \times I \stackrel{C}{\leftrightarrow} B \times J$).

Given $A \subset V$, setting 
$$\xi^A_t(x) = \mathds{1}\{A \times \{0\} \leftrightarrow (x,t)\}\qquad t \geq 0,$$
we obtain a Markov process $(\xi^A_t)_{t \geq 0}$ with $\xi^A_0 = \mathds{1}_A$ and the same distribution as the process given by the generator \eqref{eq:generator}. We will always assume that the contact process is constructed in this  way.

As mentioned in the introduction, we denote by $\underline{0}$ and $\underline{1}$ the configurations which are identically 0 and 1, respectively, and define the extinction time $\uptau_G = \inf\{t: \xi^{\underline{1}}_t  = \underline{0}\}$.

\subsection{Some preliminary results about graphs and the contact process}
\label{sec.basictool}
We will now state a few results concerning graphs and the contact process on line segments and star graphs. These results will be the basic tools in our proofs.

The first two results are not new, but for the sake of completeness we sketch their proof, as they are short and elementary.  
\begin{lemma}\label{lem:split} (i) (Lemma 3.1 in \cite{mmvy13}) Let $n, d \in \N$ with $d < n$. If $T$ is a tree of size $n$ in which all vertices have degree bounded by $d$, then there exists an edge whose removal separates $T$ into two subtrees $T_1$ and $T_2$ both of size at least $\lfloor n/d\rfloor$. \\
(ii) If $T$ is a tree of size $n$, $T$ has a vertex $x$ such that the subgraphs attached to $x$ all have size smaller than or equal to $|T|/2$.
\end{lemma}
\begin{proof} To prove (i), suppose the result is not true for some tree $T$. 
Consider an edge $\{x,y\}$, whose removal separates $T$ into two subtrees $T_x$ and $T_y$, attached respectively to $x$ and $y$, 
with the largest one being of minimal size among all edges of $T$. 
Assume for instance that $|T_x|\ge |T_y|$. Our starting hypothesis on $T$ implies then that $ |T_y|\le \lfloor n/d\rfloor-1$. 
Moreover, by definition of the edge $\{x,y\}$ all subtrees attached to $x$ must have size bounded by $n/2$, and thus even by $\lfloor n/d\rfloor-1$, using again our hypothesis on $T$. But since $x$ is of degree smaller than $d$, we deduce that 
$n=|T_y| + |T_x| \le (\lfloor n/d\rfloor-1) + 1+ (d-1) ( \lfloor n/d\rfloor - 1) < n$, and a contradiction. 

For (ii), choose any vertex in $T$, and call it $x_0$. 
If (by chance) all the subgraphs attached to $x_0$ have size bounded by $|T|/2$, there is nothing more to do. 
If not, one of them, call it $T_1$, has size larger than $|T|/2$. Call $x_1$ the only neighbor of $x_0$ in $T_1$. If all subgraphs attached 
to $x_1$ have size bounded by $|T|/2$, we are done, and if not one of them, say $T_2$, has size larger than $|T|/2$. 
Then the only thing to observe is that it cannot be the component containing $x_0$, as this one has size $|T\backslash T_1|$, which by definition of $T_1$ is smaller than $|T|/2$. Therefore, if we call $x_2$ the only neighbor of 
$x_1$ in $T_2$, we have $x_2\neq x_0$. Now we can continue like this, defining a sequence of vertices $(x_i)$, until we find a convenient vertex, and this has to happen, since the $(x_i)$ are all distinct and the graph is finite. 
\end{proof}

\begin{lemma} \label{lem:attract} 
(i) (Lemma 4.5 in \cite{mmvy13}) For any graph $G$,
\begin{equation}\label{eq:attract_geom}\P[\uptau_G \leq t] \leq \frac{t}{\E[\uptau_G]} \qquad \text{for all }t \geq 0.\end{equation}
(ii) For any graph $G$ with $n$ vertices and $m$ edges,
\begin{equation} \label{eq:upper_bound_E}\E\Ll[\uptau_G\Rr]\leq e^{n+2\lambda m}.\end{equation}
\end{lemma}
\begin{proof} (sketch) The first statement follows from the fact that, for any $t > 0$, by attractiveness of the contact process, $\uptau_G$ is stochastically dominated by $t \cdot Y$, where $Y$ is a random variable with geometric distribution with parameter $\P[\uptau_G \leq t]$. The second statement follows from observing that, in each unit time interval, with probability $e^{-n-2\lambda m}$ there is a recovery mark in each vertex of $G$ and no transmission along any of the edges of $E$.
\end{proof}

\vspace{0.2cm}
\noindent The next two lemma are part of the folklore now. In particular Lemma \ref{lem:infest_segment} was already used in \cite{mmvy13} (see Proposition 2.1 thereof), but without a full proof, so for convenience of the reader we provide one in the appendix.     
\begin{lemma}\label{lem:infest_segment}
There exists a constant $c_{\text{line}} > 0$, such that for any $n $, the contact process on the line segment $\{0,\ldots, n\}$ satisfies:
\begin{align}
\label{eq:connect_segment}&\P\Ll[(0,0) \leftrightarrow (n,t) \text{ for some } t \leq n/c_\text{line}\Rr] > c_\text{line};\\[.2cm]
&\label{eq:2path_line} \P\Ll[\text{there exists $x$ such that }\xi_t^x \neq \underline 0\text{ and } \xi_t^x \neq \xi_t^{\underline 1}  \Rr] < e^{-c_\text{line} \cdot n} \qquad \text{for all } t  \geq n/c_\text{line};\\[.2cm]
& \E\Ll[\uptau_{\{0,\ldots, n\}}\Rr] \ge e^{c_\text{line} \cdot n}. \label{eq:exp_line}
\end{align}
\end{lemma}

\begin{lemma}\label{lem:infest_star}
There exists a constant $c_{\text{star}} > 0$ such that, for any $n\ge 2$, the contact process on the star graph $S_n$ of size $n$ satisfies:
\begin{align}
& \label{eq:star_survive} \text{for any $x$, }\P \Ll[\xi_n^x \neq \underline 0 \Rr]>c_{\text{star}};\\[.2cm]
&\label{eq:2path_star} \P\Ll[\text{there exists $x$ such that }\xi_t^x \neq \underline 0\text{ and }  \xi_t^x \neq \xi_t^{\underline 1}  \Rr] < e^{-c_\text{star} \cdot n} \qquad \text{for all } t  \geq n;\\[.2cm]
& \E\Ll[\uptau_{S_n}\Rr] \ge e^{c_\text{star} \cdot n}. \label{eq:exp_star}
\end{align}
\end{lemma}
Let $F$ be either a line segment or a star of size $n$. We say 
that $F$ is \textit{lit} in configuration $\xi \in \{0,1\}^F$, or simply that $\xi$ is lit, if 
$$\P\Ll[\xi_{\exp(c_{0}\cdot n)}\neq \underline 0 \mid \xi_0=\xi\Rr]>1-e^{-c_{0}\cdot n},$$ 
with $c_{0}=\min(c_{\text{line}},c_{\text{star}})/3$.
The previous results imply the following: 
\begin{corollary}
\label{cor.lit}
Let $F$ be either a line segment or a star graph of size $n$. Then 
\begin{itemize}
\item[(i)] The fully occupied configuration $\underline 1$ is always lit. 
\item[(ii)] If $F$ is lit in a configuration $\xi$, then 
\begin{equation}
\label{eq:lit}
\P\Ll[F \text{ is lit in configuration }\xi_t\mid \xi_0=\xi\Rr] > 1- 4e^{- c_{0}\cdot n} \qquad \text{for all } t  \in [n/c_0,e^{c_{0}\cdot n}].
\end{equation}
\item[(iii)] Let $\tilde c_{0}= \min(c_{\text{line}}^2,c_{\text{star}})$. Then for any $x$, 
\begin{equation}
\label{eq:litx}
\P\Ll[F \text{ is lit in configuration }\xi_t^x\Rr] > \tilde c_{0} - e^{-c_{\text{star}}\cdot n}-4e^{- c_{0}\cdot n} \ \text{for all } t  \in [n/c_0,e^{c_0\cdot n}].
\end{equation}
\end{itemize}
\end{corollary}
\begin{proof} Part (i) is a direct consequence of Lemma \ref{lem:attract} (i), \eqref{eq:exp_line} and \eqref{eq:exp_star}. For the second part, assume that $F$ is lit in some configuration $\xi$, and denote by $A$ the set of configurations which are not lit. Note first that 
\begin{align}
\label{eq:xitA}
\nonumber \P\Ll[\xi_t\in A\mid \xi_0=\xi\Rr] & \le \P\Ll[\xi_t^{\underline 1}\in A\Rr]+\P\Ll[\xi_t\neq \xi_t^{\underline 1}, \, \xi_t\neq \underline 0\mid \xi_0=\xi\Rr]  
+ \P\Ll[\xi_t=\underline 0 \mid \xi_0=\xi\Rr] \\
&\le \P\Ll[\xi_t^{\underline 1}\in A\Rr]+e^{-3c_{0}\cdot n}+ e^{-c_{0}\cdot n},
\end{align}
where for the last inequality we have used \eqref{eq:2path_line} and \eqref{eq:2path_star} for the second term and the definition 
of a lit configuration for the last term.  Now by using Lemma  \ref{lem:attract} (i) and the Markov property, we get 
\begin{align}
\nonumber 2e^{-2c_{0}\cdot n} \ge & \P\Ll[\uptau_F\le t+e^{c_{0}\cdot n}\Rr] \ge \P\Ll[\uptau_F\le t+e^{c_{0}\cdot n},\, \xi_t^{\underline 1}\in A\Rr]\\
&\label{eq:xitA2} \ge e^{-c_{0}\cdot n}\cdot \P\Ll[\xi_t^{\underline 1}\in A\Rr].
\end{align}
The result follows by combining \eqref{eq:xitA} and \eqref{eq:xitA2}. For Part (iii), note first that if $F$ is a line segment, then 
$$\left\{(x,0)\leftrightarrow (0,t) \text{ and }(x,0)\leftrightarrow (n,t)\right\}\subset \left\{\xi_t^x\neq \underline 0,\, \xi_t^x = \xi_t^{\underline 1}\right\}.$$
Therefore by combining \eqref{eq:connect_segment}, together with Part (i) and (ii), we deduce the result for a line segment. Likewise if $F$ is a star graph the result follows from \eqref{eq:star_survive}, \eqref{eq:2path_star}, together with Part (i) and (ii).  
\end{proof}

\subsection{A coupling result and consequences}
\label{sec.maintool}
The next proposition is the coupling result discussed already in the introduction. 
\begin{proposition}\label{prop:prelim_coup}
There exists $c_\text{coup} > 0$, such that for any $n\ge 2$ and any tree $G$ with $n$ vertices,
$$\P\Ll[\xi^A_t \neq \underline{0},\;\xi^A_t \neq \xi^{\underline{1}}_t\Rr] \leq \exp\left\{-c_\text{coup}\cdot  \left \lfloor \frac{t}{n (\log n)^3}\right\rfloor\right\} \quad \text{for all } t \geq 0 \text{ and } A \neq \varnothing.$$
\end{proposition}
This is an immediate consequence of the following lemma.
\begin{lemma}\label{lem:prelim_coup}
There exists $c_1 < 1$ such that, for any tree $G$ with $n$ vertices and any $t \geq n(\log n)^3$,
\begin{equation}\label{eq:prelim_coup}
\P\Ll[\xi^A_{t} \neq \underline{0},\;\xi^A_{t} \neq \xi^{\underline{1}}_{t}\Rr] < c_1 \quad \text{for all } A \neq \varnothing.
\end{equation} 
\end{lemma}
\begin{proof}
It is sufficient to find $c_1$ such that \eqref{eq:prelim_coup} holds for $n$ large enough, as we can then make $c_1$ approach 1, if necessary, to take care of the remaining values of $n$.

\noindent If $|G| = n$, then $G$ necessarily has a subgraph $G_0$ which is either a star or a line segment and satisfies
$$|G_0| \geq \max\left( \sqrt{\log n},\; \text{diam}(G)\right).$$
Let $\bar{c} = c_0\cdot \tilde c_{0}$, and
$$t_1 = \frac{16\cdot |G_0|}{\bar{c}^2} ,\qquad t_2 = t_1 + \frac{|G_0|}{c_{0}},\qquad t_3 = t_2 + \frac{16\cdot |G_0|}{\bar{c}^2}.$$ 
Fix an arbitrary nonempty subset $A$ of  $G$. Note that, by \eqref{eq:connect_segment} and \eqref{eq:litx},
$$\P\Ll[G_0 \text{ lit in }  \xi^A_{2|G_0|/\bar{c}} \Rr] > c_{\text{line}}\cdot (\tilde c_{0}-4e^{- c_{0}\cdot n})\ge \bar{c},$$
when $n$ is large enough. 
Then, by the Markov property, we also have
$$\P\Ll[\xi^A_{t_1} \neq \underline{0},\; G_0 \text{ not lit in } \xi^A_t \text{ for any } t \leq t_1\Rr] \leq (1-\bar{c})^{\left\lfloor \frac{t_1}{2|G_0|/\bar{c}}\right\rfloor}.$$
By definition of being lit for a configuration, we then get
\begin{equation}\label{eq:bound_A}
\P\Ll[\begin{array}{c}\xi^A_{t_2} \neq \underline{0},\; \nexists \{y,z\} \subset G_0 \text{ such that } \\[.2cm]A\times\{0\} \leftrightarrow (y,t_1) \stackrel{G_0}{\leftrightarrow} (z,t_2)\end{array}\Rr]\\
\leq (1-\bar{c})^{\left\lfloor \frac{t_1}{2|G_0|/\bar{c}}\right\rfloor} + e^{-c_{0}\cdot |G_0|}<\frac{1}{256},
\end{equation}
when $n$ is large enough.

\noindent Let now $K = \lfloor (\log n)^2 \rfloor$ and define the times
$$ s_k = t_3 \cdot k,\quad s_k' = s_k + t_1,\quad s_k'' = s_k + t_2,\quad k \in \{0,1,\ldots,K\}.$$
Define also the events
\begin{align*}
&E^{G_0}_k = \left\{\begin{array}{c}\text{for all } x,y,z,w \in G_0 \text{ with } (x, s_k') \stackrel{G_0}{\leftrightarrow} (y, s_k'') \text { and }\\[.2cm] (z, s_k') \stackrel{G_0}{\leftrightarrow} (w, s_k''),\text{ we have } (x,s_k') \stackrel{G_0}{\leftrightarrow} (w,s_k'') \end{array}\right\} \qquad k\leq K,\\[.2cm]
&E_k^x = \left\{\text{for some } y \in G_0,\; (x,0) \leftrightarrow (y,s_k')  \stackrel{G_0}{\leftrightarrow} G_0 \times \{s_k''\} \right\},\\[.2cm]
&\hat E^x_k =\left\{\text{for some } y \in G_0,\; G_0 \times \{s_{k}'\} \stackrel{G_0}{\leftrightarrow} (y,s_{k}'') \leftrightarrow (x,s_K)\right\} \qquad 0\leq k\leq K-1,\; x \in G.
\end{align*}
For any $x \in G$ and $k_1,\ldots, k_m$ with $0\leq k_1 < \cdots < k_m \leq K-1$, we have 
$$\begin{aligned}&\P\Ll[\{\xi^x_{s_K} \neq \underline{0}\} \cap \bigcap_{j=1}^m(E^x_{k_j})^c \Rr] \leq \P\Ll[\{\xi^x_{s_{k_m + 1}} \neq \underline{0}\} \cap \bigcap_{j=1}^m(E^x_{k_j})^c \Rr] \\[.2cm]
&\leq \sum_{A \neq \varnothing} \P\Ll[\{\xi^x_{s_{k_m}} = A\} \cap \bigcap_{j=1}^{m-1}(E^x_{k_j})^c \Rr] \cdot \P\Ll[\begin{array}{c}\xi^A_{t_2} \neq \underline{0},\; \nexists \{y,z\} \subset G_0 \text{ such that } \\[.2cm]A\times\{0\} \leftrightarrow (y,t_1) \stackrel{G_0}{\leftrightarrow} (z,t_2)\end{array}\Rr] \\[.2cm]
&\stackrel{\eqref{eq:bound_A}}{\leq} \frac{1}{256}\cdot \P\Ll[\{\xi^x_{s_{k_m}} \neq \underline{0}\}\cap \bigcap_{j=1}^{m-1}(E^x_{k_j})^c \Rr].
\end{aligned}$$
Iterating, we get
\begin{equation*}\P\Ll[\{\xi^x_{s_K} \neq \underline{0}\} \cap \bigcap_{j=1}^m(E^x_{k_j})^c \Rr] \leq \left(\frac{1}{256}\right)^m.\end{equation*}
Thus,
\begin{align}\begin{split}
\P\Ll[\bigcup_{x\in G} \left\{\xi^x_{s_K} \neq \underline{0},\;\sum_{k=0}^{K-1} \mathds{1}_{(E^x_k)^c} > K/4\right\}\Rr] \leq \frac{n\cdot  |\{I \subset \{0,\ldots, K-1\}: |I| = \frac{K}{4}\}|}{256^{K/4}} \\ < \frac{n\cdot 2^K}{256^{K/4}} = \frac{n}{2^K}.\end{split}
\label{eq:Esbound}\end{align}
Similarly,
\begin{equation}\label{eq:hatEsbound}
\P\Ll[\bigcup_{x\in G} \left\{G \times \{0\} \leftrightarrow (x,s_K),\;\sum_{k=0}^{K-1} \mathds{1}_{(\hat E^x_k)^c} > K/4\right\}\Rr] <  \frac{n}{2^K}.
\end{equation}
Then by \eqref{eq:2path_line} and \eqref{eq:2path_star}, we get 
\begin{equation}\label{eq:Fsbound}
\P\Ll[ \sum_{k=0}^{K-1} \mathds{1}_{(E^{G_0}_k)^c}> K/4\Rr] \leq 2^K (e^{-\bar{c}|G_0|})^{K/4}.
\end{equation}
Now defining
$$\begin{aligned}&E^x = \left\{\xi^x_{s_K} = \underline{0}\right\} \cup \left\{\xi^x_{s_K} \neq \underline{0},\;\sum_{k=0}^{K-1} \mathds{1}_{(E^x_k)^c} \leq K/4\right\},\\[.2cm]
&\hat E^x = \left\{G \times \{0\} \nleftrightarrow (x,s_K)\right\} \cup   \left\{G \times \{0\} \leftrightarrow (x,s_K),\;\sum_{k=0}^{K-1} \mathds{1}_{(\hat E^x_k)^c} \leq  K/4\right\}\qquad x \in G,\\[.2cm]
&E^{G_0} = \left\{\sum_{k=0}^{K-1} \mathds{1}_{(E^{G_0}_k)^c}\leq K/4\right\},
\end{aligned}$$
we see that \eqref{eq:Esbound}, \eqref{eq:hatEsbound} and \eqref{eq:Fsbound} imply that there exists some $n_0 \in \N$ such that, if $n \geq n_0$, 
$$\P\Ll[\left(\bigcup_{x\in G} (E^x \cap \hat E^x)^c\right) \cup (E^{G_0})^c \Rr] < \frac12.$$
We claim now that, for any $A \neq \varnothing$,
$$\bigcap_{x \in G} (E^x \cap \hat E^x) \cap E^{G_0} \subseteq\{\xi^A_{s_K} = \underline{0}\} \cup\{\xi^A_{s_K} = \xi^{\underline{1}}_{s_K}\}.$$
Indeed, assume that the event on the left-hand side occurs and $\xi^A_{s_K} \neq \underline{0}$. Then, there exists $x \in A$ such that $\xi^{x}_{s_K} \neq \underline{0}$. Fix $y$ such that $\xi^{\underline{1}}_{s_K}(y) = 1$, that is, $G \times \{0\} \leftrightarrow (y,s_K)$. Since by assumption
$$\sum_{k=0}^{K-1} \mathds{1}_{(E^{x}_k)^c} \leq \frac{K}{4},\quad \sum_{k=0}^{K-1} \mathds{1}_{(\hat E^{ y}_k)^c} \leq \frac{K}{4},\quad \sum_{k=0}^{K-1} \mathds{1}_{(E^{G_0})^c} \leq \frac{K}{4},$$
there exists $k^*$ such that $E^{x}_{k^*}$, $\hat E^{y}_{k^*}$ and $E^{G_0}_{k^*}$ all occur. We then have, for some $x', x'',y',y'' \in G_0$,
$$( x,0)\leftrightarrow (x', s_{k^*}') \stackrel{G_0}{\leftrightarrow} (x'',s_{k^*}''),\qquad (y',s_{k^*}') \stackrel{G_0}{\leftrightarrow}(y'', s_{k^*}'') \leftrightarrow (y,s_K).$$
Thus 
$$(x,0) \leftrightarrow (x', s_{k^*}') \leftrightarrow (y'', s_{k^*}'') \leftrightarrow (y,s_K),$$
and therefore $\xi^A_{s_K}(y) \geq \xi^x_{s_K}( y) = 1$. This proves that $\xi^A_{s_K} = \xi^{\underline{1}}_{s_K}$.

\vspace{0.2cm}
\noindent Finally, to obtain the expression \eqref{eq:prelim_coup}, note that for $n$ large enough we have $n(\log n)^3 > s_K$, so that, for any $t\geq n(\log n)^3$,
\begin{equation*}\P\Ll[ \xi^A_{t} \neq \underline{0},\;\xi^A_{t} \neq \xi^{\underline{1}}_{t}\Rr] \leq \P\Ll[ \xi^A_{s_K} \neq \underline{0},\;\xi^A_{s_K} \neq \xi^{\underline{1}}_{s_K}\Rr].\end{equation*}
\end{proof}

We now give an important application of Proposition \ref{prop:prelim_coup}, 
which says that whenever we cut a tree into disjoints connected subtrees, 
a lower bound on the mean extinction time on the original tree is obtained by taking the product of the mean extinction times on the subtrees, up to some correction factor. The latter is negligible as long as the number of pieces in the decomposition of 
the tree is not too large. Note that a similar result was proved in \cite{CMMV13}. 

\begin{proposition}
\label{lem:induction}
There exists a constant $c_\text{split} > 0$ such that, for any tree $G$ containing $N$ connected and disjoint subtrees $G_1,\ldots, G_N$,
\begin{equation}\label{eq:lemind}
\E\Ll[\uptau_G\Rr] \geq \frac{c_\text{split}}{(2|G|^3)^{N+1}} \cdot \prod_{i=1}^N \E\Ll[\uptau_{G_i}\Rr].
\end{equation}
\end{proposition}
\begin{remark}
{\em By using simply Lemma \ref{lem:attract} and no coupling argument we could have directly obtained a much weaker version of this result, namely 
$$\E[\uptau_G] \ge \frac 12 \left( \frac12 \prod_{i=1}^N \E[\uptau_{G_i}]\right)^{1/N};$$
 this would have been insufficient for the applications we have in mind. }
\end{remark}
\begin{proof}
Fix $s > 0$ and define the events
\begin{align*}
&E_k = \bigcup_{i=1}^N \{ G_i \times \{sk\} \stackrel{G_i}{\longleftrightarrow} G_i \times \{s(k+2)\}\},\\
&F_k = \left\{\begin{array}{l}\text{for all } x,y,z,w \in G \text{ with } (x,sk) \leftrightarrow (y,s(k+1)) \text{ and }\\(z,sk) \leftrightarrow (w,s(k+1)), \text{ we have } (x,sk) \leftrightarrow (w,s(k+1)) \end{array}\right\} \quad k \in \{0,1,\ldots\}.
\end{align*}
It is readily seen that
$$\bigcap_{k=0}^K (E_k \cap F_k) \subseteq \left\{ \xi^{\underline{1}}_{s(K+1)} \neq \underline{0}\right\}.$$
By \eqref{eq:attract_geom} and  Proposition \ref{prop:prelim_coup},
\begin{align}\label{remem1}&\P\Ll[E_k^c\Rr]\leq \frac{(2s)^N}{\prod_{i=1}^N \mathbb{E}[\uptau_{G_i}]},\\[.2cm]&  \label{remem2}\P\Ll[F_k^c\Rr]\leq \sum_{x \in G} \P\Ll[\xi^x_s \neq \varnothing,\;\xi^x_s \neq \xi^{\underline{1}}_s \Rr] \leq |G|\cdot \exp \left\{ -c_\text{coup} \cdot \left\lfloor \frac{s}{|G|(\log |G|)^3} \right \rfloor\right\}.
\end{align}
Then, for any $t > s$,
\begin{equation} \label{eq:split_main}\P\Ll[\uptau_G \leq t\Rr] \leq \left\lceil\frac{t}{s}\right \rceil \left( \frac{(2s)^N}{\prod_{i=1}^N \E\Ll[\uptau_{G_i}\Rr]} + |G|\cdot \exp \left\{ -c_\text{coup} \cdot \left\lfloor \frac{s}{|G|(\log |G|)^3} \right \rfloor\right\}\Rr).\end{equation}
Let $s = |G|^3$ and $t = \frac{1}{(2|G|^3)^N}\cdot  \prod_{i=1}^N \mathbb{E}[\uptau_{G_i}]$. In case we have $t \leq s$, then also
$$\frac{\prod_{i=1}^N \E[\uptau_{G_i}]}{(2|G|^3)^{N+1}} \leq \frac{t}{s} \leq 1,$$
so \eqref{eq:lemind} holds trivially, since $\E\left[ \uptau_G\right] \geq 1$ for any graph $G$. Now, if $t > s$, using the inequality
$$\prod_{i=1}^N \E[\uptau_{G_i}]  \stackrel{\eqref{eq:upper_bound_E}}{\leq} e^{(2\lambda + 1)|G|},$$
we see that the right-hand side of \eqref{eq:split_main} is smaller than $1/2$ when $|G|$ is large enough. This proves the result for $|G|$ large enough, with $c_\text{split} = 1/2$. We can then reduce the value of $c_\text{split}$ to take care of the remaining cases.
\end{proof}

We will encounter situations in which the above proposition is not useful because the sets $G_1,\ldots, G_N$ are too small compared to $G$, so that the denominator on the right-hand side of \eqref{eq:lemind} is too large compared to the numerator. In case we can guarantee that the distances between the $G_i$'s are not too large, the following can then be valuable.
\begin{proposition}\label{prop:new_bound}
If $G$ is a tree containing $N$ disjoint connected subtrees $G_1,\ldots, G_N$ and $0 < s < t$,
\begin{equation}\label{eq:new_bound}
\P[\uptau_G \leq t] \leq \left \lceil \frac{t}{s} \right \rceil \cdot \left( \frac{(2s)^N}{\prod_{i=1}^N \E[\uptau_{G_i}]} + \sum_{1 \leq i < j \leq N}\sigma_{i,j} \cdot \exp\left\{-c_\text{coup} \cdot \left \lfloor \frac{s}{\sigma_{i,j} (\log \sigma_{i,j})^3}\right \rfloor\right\} \right),
\end{equation}
where $\sigma_{i,j} = |G_i| + |G_j| + \mathrm{dist}(G_i,G_j) - 1$.
\end{proposition}
\begin{proof}
For each distinct $i$ and $j$, define $G_{i,j}$ as the connected graph obtained as the union of $G_i$, $G_j$ and the shortest path between $G_i$ and $G_j$. Note that $|G_{i,j}| = \sigma_{i,j}$. For each $k\in\{0,1,\ldots\}$, define $E_k$ exactly as in the proof of Proposition \ref{lem:induction}, and define
$$\tilde F_k = \bigcap_{1 \leq i < j \leq N}  \left\{\begin{array}{l}\text{for all } x,y,z,w \in G_{i,j} \text{ with } (x,sk) \stackrel{G_{i,j}}{\leftrightarrow} (y,s(k+1)) \text{ and }\\[.2cm](z,sk) \stackrel{G_{i,j}}{\leftrightarrow} (w,s(k+1)), \text{ we have } (x,sk) \stackrel{G_{i,j}}{\leftrightarrow} (w,s(k+1)) \end{array}\right\}.$$
Then, $$\bigcap_{k=0}^K (E_k \cap \tilde F_k) \subseteq \left\{ \xi^{\underline{1}}_{s(K+1)} \neq \underline{0}\right\},$$
so the desired inequality follows from bounding as in \eqref{remem1} and \eqref{remem2}.
\end{proof}

\section{Proofs of main results}
\subsection{Level 1: a polynomial lower bound}
\label{sec.lev1}
\begin{proposition}
\label{prop:lev1} There exists $n_1 \in \mathbb{N}$ such that, if $n \geq n_1$ and $G$ is a tree with $n$ vertices, then
$\mathbb{E}\left[\uptau_G\right] \geq n^{12}.$
\end{proposition}
\begin{proof}
\noindent Let $C = 4/c_{0}$. If $G$ contains a star graph or a line segment of size larger than $C \log n$, then \eqref{eq:exp_line} and \eqref{eq:exp_star} imply that $\E[\uptau_G]\ge n^{12}$. 

\noindent Assume that both the maximum degree and the diameter of $G$ are smaller than $C\log n$. Using Lemma \ref{lem:split}, we can find two disjoint connected subgraphs $H_1, H_1'$ so that $$G = H_1 \cup H_1',\quad \quad|H_1| \geq \left \lfloor \frac{n}{C\log n} \right \rfloor \quad \text{ and } \quad |H_1'| \geq n/2.$$ Applying Lemma \ref{lem:split} again, we can further split $$H_1' = H_2 \cup H_2', \quad\quad  |H_2| \geq \left\lfloor \frac{n}{2C\log n} \right\rfloor  \quad \text{ and }\quad  |H_2'| \geq n/4.$$ By continuing this procedure for \begin{equation} \label{eq:choose_N} N := \lfloor (\log n)^{3/4} \rfloor \end{equation} times, we obtain disjoint connected subgraphs $H_1, \ldots, H_N$ with
$$|H_i| \geq \left\lfloor \frac{n}{2^{i-1}C\log n} \right \rfloor \geq \sqrt{n}\qquad  i=1,\ldots, N$$
(assuming $n$ is large enough).
Since each $H_i$ has both maximum degree and diameter smaller than $C \log n$, we can find subgraphs $G_i \subset H_i$ of size $\lfloor\sqrt{\log n}\rfloor $ which are either stars or line segments. By \eqref{eq:exp_line} and \eqref{eq:exp_star}, we have \begin{equation} \label{eq:bound_each}\E[\uptau_{G_i}] \geq \exp\{c_0 \sqrt{\log n}\} \text{ for each } i. \end{equation}

We now want to apply Proposition \ref{prop:new_bound} to $G$ and its subgraphs $G_1, \ldots, G_N$. Letting $\sigma_{i,j}$ be as in \eqref{prop:new_bound}, we have \begin{equation} \label{eq:bound_sigma}\sigma_{i,j} \leq 2\sqrt{\log n} + \mathrm{diam}(G) \leq 2C\log n,\end{equation}
so, letting $s = (\log n)^3$ and $t = 2n^{12}$ and using  \eqref{eq:choose_N}, \eqref{eq:bound_each} and \eqref{eq:bound_sigma}, the right-hand side of \eqref{eq:new_bound} is smaller than 
\begin{align*}
&\left \lceil \frac{2n^{12}}{(\log n)^3} \right \rceil \cdot \left( (2(\log n)^3)^{(\log n)^{3/4}}\cdot \exp\{-c_0 (\log n)^{5/4} \}\right. \\ &\hspace{3cm}\left. + (\log n)^{3/2}\cdot  2C\log n \cdot \exp\left\{-c_\text{coup} \cdot \left \lfloor \frac{(\log n)^3}{2C\log n \cdot (\log (2C\log n))^3}\right \rfloor\right\} \right),
\end{align*}
which is in turn smaller than $1/2$ when $n$ is large enough. 

\end{proof}

\subsection{Proof of Theorem \ref{thm:couple}}
\label{sec.couple}
According to Lemma A.1 in \cite{mmvy13} and Lemma \ref{lem:attract}, all we have to prove is that there exists a sequence $(a_n)$ such that $a_n = o(\E[\uptau_{G_n}])$ and for any $v\in G$, 
$$\P\Ll[\xi_{a_n}^v\neq \underline{0},\, \xi_{a_n}^v\neq \xi^{\underline{1}}_{a_n}\Rr] = o(1).$$
But this readily follows from Propositions \ref{prop:prelim_coup} and \ref{prop:lev1}.

\subsection{Level 2: a stretched exponential lower bound with exponent $1/3$}
\label{sec.lev2}
\begin{proposition}
\label{prop:lev2} There exists $n_2 \in \N$ such that, if $n \geq n_2$ and $G$ is a tree with $n$ vertices, then $\E[\uptau_G] > \exp\{c_0\cdot n^{1/3}\}$, where $c_0$ is as in Corollary \ref{cor.lit}.
\end{proposition}
\begin{proof}
\noindent Let $N = \lfloor n^{1/3} \rfloor$. If $G$ contains a subgraph with more than $N$ vertices which is either a star graph  or a line segment, then \eqref{eq:exp_line} and \eqref{eq:exp_star} give the desired result.

\noindent Now assume that the maximum degree and diameter of $G$ are both bounded by $N$; we can then repeatedly split $G$ using Lemma \ref{lem:split} and obtain disjoint connected subgraphs $G_1,\ldots, G_N$, all with at least $N$ vertices. If $n$ is large enough that $N$ is larger than the constant $n_1$ of Proposition \ref{prop:lev1}, we have $\E[\uptau_{G_i}] \geq |G_i|^{12} \geq n^{4}/2$ for each $i$. Then, by Proposition \ref{lem:induction},
$$\mathbb{E}[\uptau_G] \geq \frac{c_\text{split}}{(2n^3)^{N+1}} \cdot \prod_{i=1}^N \mathbb{E}[\uptau_{G_i}] \geq \frac{c_\text{split}}{2^{2n^{1/3}}\cdot n^{3n^{1/3}+3}} \cdot (n^4/2)^{n^{1/3}} > e^{n^{1/3}},$$
if $n$ is large enough.
\end{proof}

\subsection{A new proof of Theorem \ref{thm:ext_bound} (ii)}
\label{sec.extbound}
In this subsection we fix some integer $d\ge 1$, and only consider graphs (in fact trees) with maximal degree bounded by $d$. Set for $r\ge 2$
$$\alpha_r:= \inf_{2\le |G|\le 2^r} \frac{\log \E[\uptau_G]}{|G|}.$$
All we have to prove is that $\alpha_r$ is bounded away from zero for $r$ large enough. 
So let $r\ge 2$ be given, and consider some graph $G$ with $2^r<|G|\le 2^{r+1}$. By using Lemma \ref{lem:split}, 
we can split $G$ in at most $d+1$ disjoint connected subgraphs of size at most $2^r$, at least if $r$ is large enough. So we can assume 
that there is a decomposition of $G$ as 
$$G=G_1\cup \dots \cup G_N,$$
with $N\le d+1$ and $|G_i|\le 2^r$, for all $i$. 
Then by using Lemma \ref{lem:induction}, we deduce that there exists a constant $C>0$ such that  
\begin{eqnarray*}
\log \E[\uptau_G] &\ge & \log \E[\uptau_{G_1}] +\dots +\log \E[\uptau_{G_N}]-C\log |G| \\
&\ge & \alpha_r |G| - C (r+1)\log 2.
\end{eqnarray*}
Since this holds for any $G$ with size bounded by $2^{r+1}$, we get the important relation: 
$$\alpha_{r+1} \ge \alpha_r - C(r+1)2^{-r}\log 2.$$
It follows by induction that for any $r_0$, 
\begin{eqnarray}
\label{induction.d}
\alpha_r\ge \alpha_{r_0} - C'r_02^{-r_0} \quad \textrm{for all }r\ge r_0, 
\end{eqnarray}
for some constant $C'>0$. Moreover, Proposition \ref{prop:lev2} shows that
\begin{eqnarray}
\label{poly.d}
\alpha_r \ge c_0 2^{-\frac23(r+1)},
\end{eqnarray}
for $r$ large enough. 
By combining \eqref{induction.d} and \eqref{poly.d}, we see that there exists $r_0$ such that 
$$\alpha_r\ge \alpha_{r_0}/2 \quad   \textrm{for all }r\ge r_0, $$ 
proving Theorem \ref{thm:ext_bound}.

\subsection{Level 3: an exponential bound with a logarithmic correction}
\label{sec.lev3}
\begin{proposition}
\label{prop:lev3} There exists $n_3 \in \N$ such that, if $n \geq n_3$ and $G$ is a tree with $n$ vertices, then $\mathbb{E}[\uptau_G] \geq \exp\{n/(\log n)^{10}\}$.
\end{proposition}
\begin{proof}
For any tree $G$ let $\beta(G) = \frac{\log \E[\uptau_G]}{|G|/(\log|G|)^{10}}$; then let
$$\beta_r = \inf_{2 \le |G| \le 2^r} \, \beta(G), \qquad r \geq 1.$$
We will be done once we prove that this sequence is bounded below by a positive constant. We start with the following claim:
\begin{claim} For any $A > 0$ and any tree $G$ at least one of the following statements holds true:
\begin{itemize}
\item $G$ has a vertex of degree at least $|G|/(\log |G|)^{10}$;
\item there exist disjoint connected subtrees $G_1,\ldots, G_N \subseteq G$ so that $|G_i| \geq \frac{1}{4}(\log |G|)^{10}$ for each $i$ and $N \geq \frac{|G|}{4A(\log|G|)^{13}}$;
\item there exists a decomposition $G = G_1 \cup \cdots \cup G_N$ of $G$ into disjoint connected subtrees with  $|G_i| \leq |G|/2$ for each $i$ and  $N \leq \frac{|G|}{A(\log|G|)^{13}}$.
\end{itemize}\label{cla:3poss}
\end{claim}
\begin{proof}
Let $G$ be a tree with degrees bounded above by $|G|/(\log |G|)^{10}$. By the second part of Lemma \ref{lem:split} there exists a decomposition of $G$ as a disjoint union of connected subgraphs: 
$$G=\{x\}\cup H_1\cup \dots \cup H_{\deg(x)},$$
with $|H_i|\le |G|/2$ for all $i$. Define
$$\cI : =\left\{i \in \{1,\ldots, \deg(x)\} :\ |H_i|\ge A(\log |G|)^{13}\right\}$$
and 
$$\cJ : = \left\{i \in \{1,\ldots, \deg(x)\} :\ \frac 14(\log |G|)^{10} \le |H_i| < A(\log |G|)^{13}\right\}.$$
Note that $$\sum_{i\in (\cI \cup \cJ)^c} |H_i| \le \frac 14 (\deg(x)) (\log |G|)^{10} \le  |G|/4.$$ 
Therefore either 
\begin{eqnarray}
\label{hypJ}
\sum_{i\in \cJ} |H_i| >  |G|/4.
\end{eqnarray}
or
\begin{eqnarray}
\label{hypI}
\sum_{i\in \cI} |H_i| >  |G|/2. 
\end{eqnarray} 
We also observe that
\begin{equation}\label{level3.cardi}
|\cI| < \frac{|G|}{A(\log |G|)^{13}}
\end{equation}
and moreover,
\begin{equation}
\label{level3.cardj}
\text{if } \eqref{hypJ} \text{ holds, then }|\cJ | \ge \frac{|G|}{4A(\log |G|)^{13}}.
\end{equation}

The second case in the statement of the lemma corresponds to \eqref{hypJ}; the graphs $G_1,\ldots, G_N$ are simply the $H_i$'s for which $i \in \mathcal{J}$ (and use \eqref{level3.cardj}). The third case corresponds to \eqref{hypI}; we let $G_1,\ldots, G_{N-1}$ be the $H_i$'s for which $i \in \mathcal{I}$ and $G_N = \{x\} \cup \left(\cup_{i \in \cI^c} H_i\right)$; then use \eqref{level3.cardi}. 
\end{proof}

\begin{claim}\label{cl:level3}
There exists $n^* \in \N$ such that, if $G$ is a tree with $|G| \geq n^*$, then
\begin{itemize}
\item[(a)] if $G$ has a vertex of degree larger than $|G|/(\log |G|)^{10}$, then $\beta(G) \geq c_\text{star}/2$;
\item[(b)] if there exist disjoint and connected $G_1,\ldots, G_N \subseteq G$ with $|G_i| \geq \frac{1}{4}(\log |G|)^{10}$ for each $i$, then $\beta(G) \geq \frac{(\log |G|)^{13}}{|G|}\cdot N$;
\item[(c)] if there exist disjoint and connected $G_1,\ldots, G_N \subseteq G$ such that $G = \cup_i G_i$, then $\beta(G) \geq \min_i \beta(G_i) - 4N\cdot \frac{(\log|G|)^{11}}{|G|}$.
\end{itemize}
\end{claim}
\begin{proof} Part (a) follows from \eqref{eq:exp_star}.

To obtain (b), assume that $|G|$ is large enough that $\frac14 (\log |G|)^{10} > n_2$, where $n_2$ is the constant of Proposition \ref{prop:lev2}, so that $\E\left[\uptau_{G_i} \right] \geq \exp\{c_0\cdot (\frac14 (\log |G|)^{10})^{1/3} \}$ for each $i$.
Then, by Proposition \ref{lem:induction} we obtain
\begin{align*}\mathbb{E}[\uptau_G] \geq \frac{c_\text{split}}{(2|G|^3)^{N+1}} \cdot \prod_{i=1}^N \mathbb{E}[\uptau_{G_i}] &\geq c_\text{split} \cdot \frac{  \exp\left\{c_0\cdot N\cdot( \frac14 (\log |G|)^{10})^{1/3} \right\} }{(2|G|^3)^{N+1}} \\[.2cm]
&\geq  c_\text{split} \cdot \left( \frac{  \exp\left\{  c_0\cdot (\frac14 (\log |G|)^{10})^{1/3} \right\} } {(2|G|^3)^{2}} \right)^N > e^{N \cdot (\log |G|)^{3} }\end{align*}
if $|G|$ is large enough. The desired estimate now follows by taking the $\log$ and dividing by $|G|/(\log |G|)^{10}$.

 Finally, for (c), using Proposition \ref{lem:induction} we obtain:
$$\begin{aligned}\log \E[\uptau_G] &\geq \sum_i \log \E[\uptau_{G_i}] + \log c_\text{split} - (N+1) \log 2 - 3(N+1)\log|G|\\
&\geq \min_i \beta(G_i)\cdot\frac{|G|}{(\log |G|)^{10}}   + \log c_\text{split} - (N+1)\log 2 - 3(N+1)\log|G|,
\end{aligned}$$
so that, when $|G|$ is large enough,
$$\log \E[\uptau_G] \geq  \min_i \beta(G_i)\cdot\frac{|G|}{(\log |G|)^{10}}   - 4N \log |G|$$
and the desired inequality follows by dividing by $|G|/(\log|G|)^{10}$. This completes the proof of Claim \ref{cl:level3}.\end{proof}

\noindent Now fix $r_0$ large enough that 
\begin{equation}\label{eq:choice_of_r0} 2^{r_0} > n^*\quad \text{ and }\quad r_0 > \frac{64}{(\log 2)^{2}}.
\end{equation} 
Then fix $A > 0$ large enough that
\begin{equation} \label{eq:choice_of_A}\frac{1}{4A} < \min\left(\frac{c_\text{star}}{2},\; \beta_{r_0}\right).\end{equation}

\noindent From Claims \ref{cla:3poss} and \ref{cl:level3} and the facts that $\frac{1}{4A} < \frac{c_\text{star}}{2}$ and $\log |G| = \log 2 \cdot \log_2 |G|$ we obtain the key inequality
\begin{equation}
\beta_{r+1} \geq \min\left( \frac{1}{4A},\; \beta_{r} - \frac{8}{A(\log 2)^{2}} \cdot \frac{1}{r^{2}}\right)\qquad \text{for all } r \geq r_0.
\end{equation}
Recall from \eqref{eq:choice_of_A} that $\beta_{r_0} > \frac{1}{4A}$; define
\begin{equation}\label{eq:def_r1}
r_1 = \inf\left\{r \geq r_0: \beta_{r} < \frac{1}{4A}\right\} - 1.
\end{equation}
If $r_1 = \infty$, then the sequence $(\beta_r)$ is bounded from below by $\frac{1}{4A}$ and we are done. Otherwise, we have $\beta_{r_1} \geq \frac{1}{4A}$ and $\beta_{r}<  \frac{1}{4A}$ for all $r > r_1$, so
$$\begin{aligned}
\beta_{r+1} \geq \min\left( \frac{1}{4A},\; \beta_{r} - \frac{8}{A(\log 2)^{2}}  \cdot \frac{1}{r^{2}}\right) =  \beta_{r} - \frac{8}{A(\log 2)^{2}}  \cdot \frac{1}{r^{2}} \quad \text{for all } r \geq r_1.
\end{aligned}$$
Using this recursively, for all $r > r_1$ we have
$$\begin{aligned}\beta_r &\geq \beta_{r_1} - \frac{8}{A(\log 2)^{2}}  \sum_{i=r_1}^\infty \frac{1}{i^{2}}\\[.2cm] &\geq \frac{1}{4A} - \frac{8}{A(\log 2)^{2}} \cdot \frac 1{r_1} \geq \frac{1}{4A} - \frac{8}{A(\log 2)^{2}} \cdot \frac 1{ r_0} \stackrel{\eqref{eq:choice_of_r0}}{\geq} \frac{1}{8A},\end{aligned}$$
completing the proof. 
\end{proof} 

\subsection{Proof of Theorem \ref{thm:ext}}
\label{sec:lev4}
\begin{proofof}\eqref{eq:main_thm}. The proof will be very similar to that of Proposition \ref{prop:lev3}. Fix $\varepsilon > 0$ and, for any tree $G$, let $\beta'(G) = \frac{\E[\uptau_G]}{|G|/(\log |G|)^{1+\varepsilon}}$. Then let
$$\beta'_r = \inf_{G: 2 \leq |G| \leq 2^r} \beta'(G),\qquad r\geq 1.$$

\begin{claim}\label{cl2:level4}
For any $A > 0$ and any tree $G$ at least one of the following statements is true:
\begin{itemize}
\item $G$ has a vertex of degree at least $|G|/(\log |G|)^{1+\varepsilon}$;
\item for some $k \in \{1,2,3\}$, there exist disjoint connected subtrees $G_1,\ldots, G_N \subseteq G$ so that $|G_i| \geq \frac14(\log|G|)^{k + \varepsilon}$ for each $i$, $\mathrm{dist}(G_i, G_j) = 2$ for each $i \neq j$ and $N \geq \frac{|G|}{12A(\log |G|)^{k+1+\varepsilon}}$;
\item there exists a decomposition $G = G_1 \cup \cdots \cup G_N$ into disjoint connected subtrees with $|G_i| \leq |G|/2$ for each $i$ and $N \leq \frac{|G|}{A(\log |G|)^{4+\varepsilon}} + 1 \leq 2\frac{|G|}{A(\log |G|)^{4+\varepsilon}}$.
\end{itemize}
\end{claim}
\begin{proof}
Fix $A > 0$. Assume that $G$ is a tree with degrees bounded above by $n/(\log n)^{1+\varepsilon}$. We again take a vertex $x$ so that all the subtrees connected to $x$, denoted $H_1,\ldots, H_{\deg(x)}$, have no more than $|G|/2$ vertices each. Now define the sets of indices
\begin{align*}
&\mathcal{I} = \{i \in \{1,\ldots, \deg(x)\}: |H_i| \geq A (\log |G|)^{4+\varepsilon}\},\\[.2cm]
&\mathcal{J}_1 = \{i \in \{1,\ldots, \deg(x)\}: \frac14 (\log |G|)^{1+\varepsilon} \leq |H_i| < (\log |G|)^{2 + \varepsilon}\},\\[.2cm]
&\mathcal{J}_2 = \{i \in \{1,\ldots, \deg(x)\}: (\log |G|)^{2+\varepsilon} \leq |H_i| < (\log |G|)^{3 + \varepsilon}\},\\[.2cm]
&\mathcal{J}_3 = \{i \in \{1,\ldots, \deg(x)\}: (\log |G|)^{3+\varepsilon} \leq |H_i| < A(\log |G|)^{4 + \varepsilon}\}.
\end{align*}
Note that $$|\mathcal{I}| \leq \frac{|G|}{A(\log |G|)^{4+\varepsilon}}.$$
Moreover, since
$$\sum_{i \in (\mathcal{I} \cup \mathcal{J}_1 \cup \mathcal{J}_2 \cup \mathcal{J}_3)^c} |H_i| \leq \deg(x) \cdot \frac14(\log |G|)^{1+\varepsilon} \leq \frac{|G|}{4},$$
at least one of the following holds:
$$(\mathrm{i})\; \sum_{i \in \mathcal{I}} |H_i| \geq \frac{|G|}{2},\qquad  (\mathrm{ii})\; \sum_{i \in \mathcal{J}_1} |H_i| \geq \frac{|G|}{12},\qquad (\mathrm{iii})\; \sum_{i \in \mathcal{J}_2} |H_i| \geq \frac{|G|}{12} ,\qquad (\mathrm{iv})\; \sum_{i \in \mathcal{J}_3} |H_i| \geq \frac{|G|}{12}.$$
We also observe that (ii), (iii) and (iv) respectively imply
\begin{align*}
& |\mathcal{J}_1| \geq \frac{|G|}{12(\log |G|)^{2 + \varepsilon}},\qquad |\mathcal{J}_2| \geq \frac{|G|}{12(\log |G|)^{3 + \varepsilon}},\qquad | \mathcal{J}_3| \geq \frac{|G|}{12A(\log |G|)^{4 + \varepsilon}}.
\end{align*}
Finally, note that the distance between $H_i$ and $H_j$ for $i \neq j$ is equal to 2, since both $H_i$ and $H_j$ are connected to $x$.
\end{proof}

\begin{claim}\label{cl:level4}
There exists $n^\star \in \N$ such that, if $G$ is a tree with $|G| \geq n^\star$, then
\begin{itemize}
\item[(a)] if $G$ has a vertex of degree larger than $|G|/(\log |G|)^{1+\varepsilon}$, then $\beta(G) \geq c_\text{star}/2$;
\item[(b)] if $k \in \{1,2,3\}$ and there exist disjoint and connected $G_1,\ldots, G_N \subseteq G$ with $|G_i| \geq \frac{1}{4}(\log |G|)^{k+\varepsilon}$ for each $i$ and $\mathrm{dist}(G_i,G_j) = 2$ for each $i \neq j$, then $\beta(G) \geq \frac{(\log |G|)^{k + 1 +\varepsilon}}{|G|}\cdot N$;
\item[(c)] if there exist disjoint and connected $G_1,\ldots, G_N \subseteq G$ such that $G = \cup_i G_i$, then $\beta(G) \geq \min_i \beta(G_i) - 4N\cdot \frac{(\log|G|)^{2+\varepsilon}}{|G|}$.
\end{itemize}
\end{claim}
\begin{proof}
The proofs of statements (a) and (c) are the same as those of (a) and (c) of Claim \ref{cl:level3}, respectively. Let us prove (b) using Proposition \ref{prop:new_bound}. In the notation of that proposition, we simply bound $\sigma_{i,j} \leq |G|$ and let $s = |G|^4$ and $t = 2\exp\{N(\log |G|)^k\}$. Note that, if $n$ is large enough, for each $i$ we have
$$\E[\uptau_{G_i}] \geq \exp \left\{\frac{\frac14 (\log|G|)^{k+\varepsilon}}{\left(\log\left(\frac14 (\log|G|)^{k+\varepsilon}\right)\right)^{10}}\right\}$$
by Proposition \ref{prop:lev3}. Then,
\begin{align*}&\P\left[\uptau_G \leq t\right] \leq  \left\lceil \frac{t}{s}\right\rceil\cdot \left(\frac{(2s)^N}{\prod_{i=1}^N \E[\uptau_{G_i}]} +N^2 \cdot \max_{i,j}\sigma_{i,j} \cdot \exp\left\{-\frac{s}{\left(\max_{i,j} \sigma_{i,j} \right)^2} \right\}\right)\\[.2cm]
&\leq  \frac{2\exp\{N(\log |G|)^k\}}{|G|^4} \left((2|G|^4)^N \exp \left\{- \frac{\frac14 (\log|G|)^{k+\varepsilon}\cdot N}{\left(\log\left(\frac14 (\log|G|)^{k+\varepsilon}\right)\right)^{10}} \right\} + N^2 |G| \exp\left\{-|G|^2\right\}\right). \end{align*}
If $n$ is large enough, this is smaller than $1/2$, uniformly on $N \in \{1,\ldots,n\}$. This shows that $\E[\uptau_G] \geq \exp\{N(\log |G|)^k\}$, so that $\beta'(G) \geq \frac{(\log |G|)^{k + 1+\varepsilon}}{|G|} \cdot N$ as desired. 
\end{proof}

Choose $r_0$ large enough that $2^{r_0} > n^\star$ and  $r_0 \geq \frac{192}{(\log 2)^2}$, and choose $A$ large enough that $\frac{1}{12A} < \min(c_\text{star}/2, \beta'_{r_0})$.

Putting together Claims \ref{cl2:level4} and \ref{cl:level4}, we obtain the inequality
$$\beta'_{r+1} \geq \min\left(\frac{1}{12A},\; \beta'_r - \frac{8}{A(\log 2)^2} \cdot \frac{1}{r^2} \right) \qquad \text{for all } r \geq r_0. $$
From here, we conclude the proof exactly as in Proposition \ref{prop:lev3}.
\end{proofof}

\begin{proofof}\eqref{eq:main_thm2}.
For every $\varepsilon > 0$ and every graph $G$ with at least two vertices, let
$$T_\varepsilon(G) = \exp\left\{\frac{|G|}{(\log|G|)^{1+\varepsilon}} \right\}.$$

\begin{claim} For every $\varepsilon > 0$ there exists $C_\varepsilon > 0$ such that, for any graph $G$, 
$$\P\left[\uptau_G \leq {T_\varepsilon(G)} \right] < C_\varepsilon \cdot T_\varepsilon(G)^{-1}. $$\end{claim}
\begin{proof} This follows from applying \eqref{eq:main_thm} with $\varepsilon$ replaced by $\varepsilon/2$ and \eqref{eq:attract_geom}.\end{proof}

\begin{claim}
For all $\varepsilon > 0$ there exists $N_\varepsilon \in \N$ such that, if $G$ is a tree and $G_0 \subset G$ is a connected subtree with $|G_0| = N_\varepsilon$, then the contact process on $G$ satisfies
$$\P\left[\xi^{G_0}_{T_\varepsilon(G)} \neq \underline{0} \right] > \frac12.$$
\end{claim}
\begin{proof}
Let $G$ be a tree with a connected subtree $G_0$. Choose a sequence of connected subtrees
$$G_0 \subset G_1 \subset \cdots \subset G_{m} = G$$
so that for each $k$, $|G_{k+1}| = |G_k| + 1$ (in particular, $m = |G| - |G_0|$). Define the events
\begin{align*}
&E_k = \left\{G_k \times \{0\} \stackrel{G_k}{\leftrightarrow} G_k \times \{T_\varepsilon(G_k)\} \right\},\qquad 0 \leq k \leq m,\\[.2cm]
&F_k = \left\{\begin{array}{l}\text{for all } x,y,z,w \in G_k \text{ with } (x,0)  \stackrel{G_k}{\leftrightarrow} (y, T_\varepsilon(G_{k-1}))\\[.2cm]\text{and }(z,0) \stackrel{G_k}{\leftrightarrow} (w, T_\varepsilon(G_{k-1})), \text{ we have } (x,0) \stackrel{G_k}{\leftrightarrow} (w, T_\varepsilon(G_{k-1})) \end{array}\right\} \qquad 1 \leq k \leq m.
\end{align*}
The desired result now follows from observing that
$$E_0 \cap \bigcap_{k=1}^m (E_k \cap F_k) \subset \left\{\xi^{G_0}_{T_\varepsilon(G)} \neq \underline{0} \right\}$$
and
$$\P[E_k^c] \leq T_\varepsilon(G_k)^{-1},\qquad \P[F_k^c] \leq \exp\left\{-c_\text{coup} \cdot \frac{T_\varepsilon(G_{k-1})}{|G_k|(\log |G_k|)^3} \right\}. $$
\end{proof}

We are now ready to conclude. Let $G$ be a tree with $|G| \geq N_\varepsilon$. Also let $A \subset G$, $A \neq \varnothing$, and $x \in A$. Fix a connected subtree $G_0 \ni x$ with $|G_0| = N_\varepsilon$. Then,
$$\P\Ll[\xi^A_{T_\varepsilon(G)} \neq \underline{0} \Rr] \ge \P\Ll[\xi^A_{1+T_\varepsilon(G)} \neq \underline{0} \Rr] \geq \P\Ll[\xi^x_1 \equiv 1 \text{ on } G_0\Rr] \cdot \frac12 \geq \frac{\theta(N_\varepsilon)}{2},$$
where we define
$$\theta(n) = \inf\left\{\P\Ll[\xi^z_1 \equiv 1 \text{ on } G'\Rr]: G' \text{ is a tree with } |G'| = n, \; z \in G' \right\}.$$
Noting that the set of pairs $(G', z)$ over which the infimum is taken is finite, and the probability is positive for each pair, we obtain $\theta(n) > 0$ for each $n$. So \eqref{eq:main_thm2} is now proved for $n$ large enough. We can now choose $c_\varepsilon$ small to cover the remaining values of $n$.
\end{proofof}

\section{Appendix: Proofs of results of Section \ref{s:prelim}}
\subsection{Proof of Lemma \ref{lem:infest_segment}}
Here we will recall some facts about the one-dimensional contact process in order to prove the two first statements of the lemma. 
The third one \eqref{eq:exp_line} is proved in \cite{lig2}, see (3.11) in Part I of that book. 

We observe that it is sufficient to prove that these statements hold for $n$ large enough, as we can then lower the value of $c_\text{line}$, if necessary, to take care of the remaining values of $n$.

We will need to simultaneously consider the contact process on the integer line $\Z$ (which we denote by $(\zeta_t)$) and on the line segment $\{0,\ldots, n\}$ (denoted by $(\xi_t)$). Our previous conventions about superscripts still apply; for example, $(\zeta^x_t)$ and $(\zeta^{\underline{1}}_t)$ are the processes on $\Z$ started respectively from only $x$ infected and full occupancy.

We first gather the results we need about the contact process on $\Z$ in the following lemma. Let $r_t = \sup\{x: \zeta^{0}_t(x) = 1\}$.  

\begin{lemma}\label{lem:good_path}
There exists $c_{\Z} > 0$ such that, for the contact process on $\Z$, 
\begin{itemize}
\item[(i)] conditioned on $\{(0,0)\leftrightarrow \infty\}$, for  large enough $z > 0$, with probability larger than $1 - e^{-c_\Z\cdot z}$ there exists an infection path $\upgamma:[0,\infty) \to \Z$ such that $\upgamma(0) = 0$ and $\upgamma(t) \geq -z+c_\Z\cdot  t$ for all $t \geq 0$;
\item[(ii)] with probability larger than $c_\Z$, there exists an infection path $\upgamma: [0,\infty) \to \Z$ such that $\upgamma(0) = 0$ and $\upgamma(t) \geq \lfloor c_\Z \cdot t \rfloor$ for all $t \geq 0$;
\item[(iii)] for large enough  $t > 0$,
\begin{equation}
\label{eq:bound_r} \P\Ll[\zeta^0_t \neq \underline{0},\;\; \max_{0 \leq s \leq t} r_s < \frac{c_\Z \cdot t}{2}\Rr] < e^{-c_\Z \cdot t}.
\end{equation}
\end{itemize}
\end{lemma}
\begin{proof}
On the event $\{(0,0)\leftrightarrow \infty\}$, define
$$\begin{aligned}&\sigma_0 \equiv 0,\quad \sigma_{n+1} = \inf\{t \geq \sigma_n + 1: (r_t, t) \leftrightarrow \infty\},\; n \geq 0,\\[.2cm]
&M_n = \max\{|x - r_{\sigma_n}|: (r_{\sigma_n}, \sigma_n) \leftrightarrow (x,t) \text{ for some } t \in [\sigma_n, \sigma_{n+1}]\}.\end{aligned}$$ 
It is shown in \cite{kuc} that, 
\begin{align}\label{eq:ren_rw}&\hspace{-.2cm}\begin{array}{l}\text{conditioned on $\{(0,0) \leftrightarrow \infty\}$, the vectors }(\sigma_{n+1}-\sigma_n,\;r_{\sigma_{n+1}} - r_{\sigma_n},\;M_n)_{n \geq 0}\\\hspace{5cm}\text{ are independent and identically distributed};\end{array}\\
&\label{eq:connect}\text{on $\{(0,0)\leftrightarrow \infty\}$, for each }n,\;(r_{\sigma_n},\sigma_n) \leftrightarrow (r_{\sigma_{n+1}}, \sigma_{n+1});\\
&\label{eq:exp_renewal}\text{there exists }\bar c > 0\text{ such that } \P\Ll[\max(\sigma_1,\;M_1) \geq m\mid (0,0)\leftrightarrow \infty\Rr] \leq e^{-\bar{c}m},\; m > 0.
\end{align}
By \eqref{eq:ren_rw} and the law of large numbers, there exist $a > 0$ and $b \in \R$ such that
\begin{equation}\label{eq:ab_lim}\P\Ll[\lim_{n \to \infty} \frac{\sigma_n}{n} = a,\; \lim_{n \to \infty} \frac{r_{\sigma_n}}{n} = b \mid (0,0) \leftrightarrow \infty\Rr] =1.\end{equation}
Moreover, by Theorem 2.19 in Chapter VI of \cite{lig1}, there exists $\alpha > 0$ such that
$$\P\Ll[\left.\lim_{t\to\infty}\frac{r_t}{t} =\alpha \;\right| (0,0) \leftrightarrow \infty\Rr] = 1,$$
so we must have $b > 0$.

\noindent Now for $z > 0$, define the event
$$E = \left\{(0,0) \leftrightarrow \infty,\; r_{\sigma_n} \geq \frac{b n}{2} - \frac{z}{3},\;  M_n \leq \frac{b n}{4} + \frac{z}{3}\text{ and } \sigma_n \leq 2a\left(n - 1 + \frac{4 z}{3b}\right)\text{ for all } n \right\}.$$
By \eqref{eq:exp_renewal}, \eqref{eq:ab_lim} and simple large deviation estimates for random walks, there exists $c > 0$ such that 
$$\P\Ll[E\mid (0,0) \leftrightarrow \infty\Rr] > 1 - e^{-cz}.$$
If $E$ occurs, by \eqref{eq:connect} we can define an infection path $\upgamma:[0,\infty) \to \Z$ such that $\upgamma(0) = 0$ and $\upgamma(\sigma_n) = r_{\sigma_n}$ for each $n$. Let $t \geq 0$. Since
$$\sigma_{\left\lceil \frac{t}{2a} - \frac{4z}{3b}\right \rceil} \leq 2a \left(\left\lceil \frac{t}{2a} - \frac{4z}{3b}\right \rceil -  1 + \frac{4z}{3b} \right) \leq t,$$
we have
$$N_t := \sup\{n: \sigma_n \leq t\} \geq \frac{t}{2a} - \frac{4z}{3b},$$
so that
\begin{equation}\label{ex:inf_path}\upgamma(t) \geq r_{\sigma_{N_t}} - M_{N_t} \geq \frac{bN_t}{2} - \frac{z}{3} - \frac{bN_t}{4} - \frac{z}{3} \geq \frac{b}{8a}\cdot t - z.\end{equation}
This proves the first statement of the lemma. 

\noindent Now fix $z > 0$ such that an infection path $\upgamma$ satisfying \eqref{ex:inf_path}  exists with positive probability. Conditioned on this, by the FKG inequality, there is a positive probability that $\zeta^0_t(0) = 1$ for all $t \in [0,8az/b]$. We can then construct an infection path $\tilde \upgamma$ such that $\tilde\upgamma(t) \geq 0$ for all $t \in [0,8az/b]$ and $\tilde \upgamma(t) = \upgamma(t)$ for all $t \geq 8az/b$. By choosing $c_\Z$ small enough (depending only on $a$ and $b$), we then have $\tilde \upgamma(t) \geq \lfloor c_\Z \cdot z\rfloor$ for all $t \geq 0$. This proves (ii).

\noindent Finally, the left-hand side of \eqref{eq:bound_r} is less than
$$\P\Ll[\zeta^0_t \neq \underline{0},\; (0,0) \nleftrightarrow \infty\Rr] + \P\Ll[\max_{0\leq s \leq t}\; r_s < \frac{c_\Z \cdot t}{2} \mid (0,0) \leftrightarrow \infty\Rr].$$
Theorem 2.30 in \cite{lig2} implies that the first term is bounded by $e^{-ct}$ for some $c >0$. 
To bound the  second term, we use Part (i) with $z = c_\Z t/2$. This completes the proof.
\end{proof}

\noindent We are now in position to prove \eqref{eq:connect_segment} and \eqref{eq:2path_line}. 

\begin{proofof}{\eqref{eq:connect_segment}:} The statement follows directly from Part (ii) of the above lemma by taking any $c_\text{line} \leq c_\Z$. \end{proofof}

\begin{proofof}{\eqref{eq:2path_line}:} 
We start observing that
\begin{equation*}\label{eq:cross_paths}
\text{if } (x,0) \stackrel{\{0,\ldots,n\}}{\leftrightarrow} \{0\} \times [0,t] \text{ and }(x,0) \stackrel{\{0,\ldots,n\}}{\leftrightarrow} \{n\} \times [0,t], \text{ then } \xi^x_t = \xi^{\underline{1}}_t.
\end{equation*}
Thus, 
\begin{align}
\label{eq:two_bounds}
\P\Ll[ \xi^x_{t} \neq \xi^{\underline{1}}_t,\; \xi^x_{t} \neq \underline{0}\Rr] \leq \P\Ll[\xi^x_t \neq \underline{0},\; \xi^x_s(n) = 0 \;\forall s \leq t\Rr] +  \P\Ll[\xi^x_t \neq \underline{0},\; \xi^x_s(0) = 0 \;\forall s \leq t\Rr].
\end{align}
We now note that
\begin{align*}
\text{if } \xi^x_t \neq \underline{0}\text{ and }\xi^x_s(n) = 0\;\forall s \leq t, \text{ then } \max\{y:\xi^x_t(y) = 1\} = \max\{y:\zeta^{x}_t(y) = 1\}.
\end{align*}
Hence,
\begin{align*}
\P\Ll[\xi^x_t \neq \underline{0},\; \xi^x_s(n) = 0 \; \forall s \leq t\Rr] &\leq \P\Ll[\zeta^x_t \neq \underline{0},\; \zeta^x_s(n) = 0\; \forall s \leq t\Rr]\\
&\leq \P\Ll[\zeta^0_t \neq \underline{0},\; \zeta^0_s(n) = 0\; \forall s \leq t\Rr]\leq \P\Ll[\zeta^0_t \neq \underline{0},\; \max_{0\leq s\leq t} r_s < n\Rr].
\end{align*}
By \eqref{eq:bound_r} and the assumption that $t \geq 2n/c_\Z$, this is less than $e^{-n}$. The same bound holds for the second term in \eqref{eq:two_bounds} by symmetry. Thus 
\begin{align*}
\P\Ll[\xi^x_t \neq \xi^{\underline{1}}_t,\;\xi^x_t \neq \underline{0}\Rr] < 2e^{-n}, 
\end{align*}
and \eqref{eq:2path_line} follows by a union bound.
\end{proofof}

\subsection{Proof of Lemma \ref{lem:infest_star}}
The result is a straightforward adaption of Lemma 3.1 in \cite{mvy13}. That lemma implies that there exists $c > 0$ such that the following holds ($o$ denotes the central vertex of the star and $\ell$ denotes Lebesgue measure on $[0,\infty)$):
\begin{equation}\label{eq:helpful_star}
\P\Ll[|\xi^A_1| \geq \frac{n}{40},\;\ell\{s \leq 1: \xi^A_s(o) = 1\} > \frac34 \Rr] > 1-e^{-cn} \text{ for all } n,\; A \subset S_n \text{ with } |A| \geq \frac{n}{40}.
\end{equation}
(The mentioned lemma is stated with the assumption that $\lambda > 1$, but the proof works equally well here). This already implies \eqref{eq:exp_star}. 

\noindent Moreover, by a straightforward computation, it can be shown that
$$\P\Ll[|\xi^o_1| \geq \frac{n}{40}\Rr] \geq \P\Ll[D^o_{[0,1]} = \varnothing,\; |\{y \neq o: D^y_{[0,1]} = \varnothing,\; D^{o,y}_{[0,1]} \neq \varnothing\}| > \frac{n}{40}\Rr] > c',$$
for some constant $c'>0$ and any $n$. Hence, for any $n$ and any set $A$ with $A \neq \varnothing$,
$$\P\Ll[|\xi^A_2| > \frac{n}{40}\Rr] > c'',$$
for some smaller constant $c''>0$. Together with \eqref{eq:helpful_star} this proves \eqref{eq:star_survive}. 

\noindent Now we prove \eqref{eq:2path_star}. To this end it is convenient to introduce the dual process: for fixed $t$ and $x$, 
the dual process $(\hat \xi^{(x,t)}_s)_{0\leq s \leq t}$ is defined by
$$\hat \xi^{(x,t)}_s(y) = \mathds{1}\{(y,t-s) \leftrightarrow (x,t)\}.$$
Recall that 
\begin{equation}
\label{dual}
\left\{\xi_t^x\neq \underline 0,\, \xi_t^x\neq \xi_t^{\underline 1}\right\}= \left\{\xi_t^x\neq \underline 0\right\} \cap \left\{\exists w:\, \hat  \xi^{(w,t)}_t \neq \underline 0,\, \xi_s^x\cap \hat  \xi^{(w,t)}_{t-s}=\emptyset \, \forall s\le t \right\}. 
\end{equation}
Now, it follows from \eqref{eq:helpful_star} that, for any $n$, $t \geq n$ and any vertex $x$,
\begin{equation}\label{eq:dual_star}\P\Ll[\xi^x_t \neq \varnothing,\; \frac{1}{n}\ell\{s \leq n: \xi^x_s(o) = 1\} \le \frac12\Rr] < e^{-c'''n}.\end{equation}
Together with a union bound this implies that, with probability larger than $1 - 2ne^{-c'''n}$, the following event occurs:
$$\begin{aligned}&\bigcap_{x \in S} \left[\left( \left\{\xi^x_t = \varnothing\right\} \cup \left\{\ell\{s \leq n: \xi^x_s(o) = 1\} >\frac{n}{2}\right\}\right) \right.\\&\hspace{2cm} \left.\cap \left( \left\{\hat \xi^{(x,t)}_t = \varnothing\right\} \cup \left\{\ell\{s \leq n: \hat\xi^{(x,t)}_{t-s}(o) = 1\} >\frac{n}{2}\right\}\right)\right].\end{aligned}$$ 
This proves \eqref{eq:2path_star}, as one can observe that the intersection of the above event with the event on the right-hand side of \eqref{dual} is empty.

\end{document}